\documentclass{amsart}
\usepackage{mathrsfs}
\usepackage{graphicx}
\usepackage{pdfsync}
\usepackage{amssymb}
\usepackage{fncylab}
\usepackage[usenames]{color}
\usepackage{footmisc}
\usepackage{mathrsfs}
\usepackage{amsthm}
\usepackage{amsmath}
\usepackage{graphicx}
\usepackage{multicol}
\usepackage[T1]{fontenc}
    \newtheoremstyle{upright}%
        {1pt plus1pt minus1pt}%
        {1pt plus1pt minus1pt}%
        {\upshape}%
        {}%
        {\bfseries\scshape}%
        {\textbf{.}}%
        {1em}%
        {}%
\newtheorem{theorem}[]{Theorem}
\newtheorem{conjecture}[theorem]{Conjecture}
\newtheorem{corollary}[theorem]{Corollary}
\newtheorem{lemma}[theorem]{Lemma}

\newtheorem{proposition}[theorem]{Proposition}

\theoremstyle{definition}
\newtheorem{definition}[theorem]{Definition}
\newtheorem{example}[theorem]{Example}
\newtheorem{notation}[theorem]{Notation}
\newtheorem{remark}[theorem]{Remark}

\renewenvironment{proof}[1][Proof]{\textbf{#1.} }{\ \rule{0.5em}{0.5em}}

\newcommand{\ks}{K\!S}

\newcommand{\sksi}[1]{\mathcal{S}(\ks_{#1})}
\newcommand{\sks}{\mathcal{S}(\ks)}
\newcommand{\omegaksi}[1]{\Omega(\ks_{#1})}
\newcommand{\omegaks}{\Omega(\ks)}
\newcommand{\ksi}[1]{\ks_{#1}}
\newcommand{\xoo}{x_0^0}

\newcommand{\xio}[1]{x_{#1}^0}
\newcommand{\yio}[1]{y_{#1}^0}
\newcommand{\xii}[2]{x_{#1}^{#2}}

\newcommand{\orbitksi}[1]{\mathscr{O}_{#1}(\xio{#1},\theta_{#1}^0)}
\newcommand{\footprintksi}[1]{\mathcal{F}_{#1}(\xio{#1},\theta_{#1}^0)}

\newcommand{\compseq}{\{\mathscr{O}_n(\xio{n},\theta^0)\}_{n=0}^\infty}

\newcommand{\compseqi}[1]{\{\mathscr{O}_n(\xio{n},\theta^0)\}_{n=#1}^\infty}
\newcommand{\compseqiang}[2]{\{\mathscr{O}_n(\xio{n},#2)\}_{n=#1}^\infty}
\newcommand{\compseqbar}{\{\mathscr{O}_n(\overline{\xio{n}},\overline{\theta^0})\}_{n=0}^\infty}

\newcommand{\orbitksiang}[2]{\mathscr{O}_{#1}(\xio{#1},#2)}
\newcommand{\orbitksixang}[3]{\mathscr{O}_{#1}(#2,#3)}

\newcommand{\cantor}{\mathscr{C}}

\newcommand{\tern}[2]{[#1,#2]}

\begin{document}



\title[Properties of hybrid orbits]{Sequences of compatible periodic hybrid orbits of prefractal Koch snowflake billiards}

\author[M. L. Lapidus]{Michel L. Lapidus}
\address{University of California, Riverside, 900 Big Springs Rd.\ Riverside, CA  92521, USA}
\email{lapidus@math.ucr.edu}
\thanks{The work of M. L. Lapidus was partially supported by the National Science Foundation under the research grants DMS-0707524 and DMS-1107750.}
\author[R. G. Niemeyer]{Robert G. Niemeyer}
\address{University of California, Riverside, 900 Big Springs Rd.\ Riverside, CA  92521, USA}
\email{niemeyer@math.ucr.edu}

\begin{abstract}
The Koch snowflake $\ks$ is a nowhere differentiable curve.  The billiard table $\omegaks$ with boundary $\ks$ is, a priori, not well defined.  That is, one cannot a priori determine the minimal path traversed by a billiard ball subject to a collision in the boundary of the table.  It is this problem which makes $\omegaks$ such an interesting, yet difficult, table to analyze.


In this paper, we approach this problem by approximating (from the inside) $\omegaks$ by well-defined (prefractal) rational polygonal billiard tables $\omegaksi{n}$.  We first show that the flat surface $\sksi{n}$ determined from the rational billiard $\omegaksi{n}$ is a branched cover of the singly punctured hexagonal torus.  Such a result, when combined with the results of \cite{Gut2}, allows us to define a sequence of compatible orbits of prefractal billiards $\omegaksi{n}$.  Using a particular addressing system, we define a hybrid orbit of a prefractal billiard $\omegaksi{n}$ and show that every dense orbit of a prefractal billiard $\omegaksi{n}$ is a dense hybrid orbit of $\omegaksi{n}$.  This result is key in obtaining a topological dichotomy for a sequence of compatible orbits.  Furthermore, we determine a sufficient condition for a sequence of compatible orbits to be a sequence of compatible \textit{periodic hybrid} orbits.

We then examine the limiting behavior of a sequence of compatible periodic hybrid orbits.  We show that the trivial limit of particular (eventually) constant sequences of compatible hybrid orbits constitutes an orbit of $\omegaks$. In addition, we show that the union of two suitably chosen nontrivial polygonal paths connects two elusive limit points of the Koch snowflake.  We conjecture that such a path is indeed the subset of what will eventually be an orbit of the Koch snowflake fractal billiard, once an appropriate `fractal law of reflection' is determined.

Finally, we close with a discussion of several open problems and potential directions for further research.  We discuss how it may be possible for our results to be generalized to other fractal billiard tables and how understanding the structures of the Veech groups of the prefractal billiards may help in determining `\textit{fractal flat surfaces}' naturally associated with the billiard flows.
\end{abstract}
\maketitle

\section{Introduction}
\label{ch:introduction}


The Koch snowflake curve $\ks$, the construction of which is depicted in Figure \ref{fig:kochconstruction75intro}, is everywhere nondifferentiable.  The absence of a well-defined tangent at any point of $\ks$ is what, a priori, prevents one from determining a billiard flow on the billiard $\omegaks$ with boundary $\ks$.  Indeed, since every point of $\ks$ is apparently a singularity of the billiard flow, one cannot a priori find a minimal path traversed by a billiard ball subject to a collision in the boundary. Our search for a solution to this problem will be, in part, motivated by the discussion on experimental results given in our earlier paper \cite{LapNie1}.

For each $n=0,1,2,...$, the prefractal $\ksi{n}$ is the $n$th (inner) polygonal approximation to $\ks$, and defines a rational polygonal billiard table $\omegaksi{n}$; that is, a polygon whose interior angles are all rational multiples of $\pi$. (See Figure \ref{fig:kochconstruction75intro}.)  Since the theory of rational polygonal billiards is very well developed (see, e.g., \cite{GaStVo,Gut1,GutJu1,GutJu2,HuSc,HaKa,KaZe,Mas,MasTa,Ta1,Ta2,Ve1,Ve2,Vo,Zo}), it is natural to want to define the dynamics on the fractal ``billiard table'' $\omegaks$ in terms of the dynamics on its prefractal approximations $\omegaksi{n}$.  The focus of this paper is then to build a foundation on which we can begin to investigate the nature of orbits of the Koch snowflake billiard $\omegaks$. We next describe the contents of this paper and outline our main results.

In order for the results of \S\S\ref{ch:TheKochSnowflakePrefractalFlatSurfaceSKSi}--\ref{sec:concludingRemarks} to be accessible to a broader audience, we provide in \S\ref{sec:BackgroundAndPreliminaries} a brief treatment of the necessary topics from the theory of mathematical billiards and particular examples from fractal geometry.  In connection with mathematical billiards, we also give a brief description of how a flat surface can be used to rigorously relate the billiard flow on the billiard $\Omega(B)$, where $B$ is a rational polygon, to the geodesic flow on the corresponding flat surface $\mathcal{S}(B)$.


\begin{figure}
	\begin{center}
\includegraphics{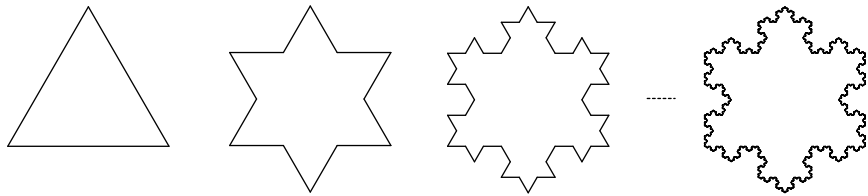}
	\end{center}
	\caption[The Koch snowflake and prefractal approximations]{The Koch snowflake curve $\ks$ and its prefractal approximations $\ksi{n}$, for $n=0,1,2,\cdots$. As is well known, $\ks$ is a fractal, nowhere differentiable and closed curve, of infinite length and enclosing a finite area.  Furthermore, it is self-similar; more precisely, it is the union of three abutting self-similar sets, each an isometric copy of the classic Koch curve.  (See, e.g., \cite{Fa}.) }
	\label{fig:kochconstruction75intro}
\end{figure}



The main results of the paper are presented in \S\S\ref{ch:TheKochSnowflakePrefractalFlatSurfaceSKSi}--\ref{sec:nontrivialPolygonalPaths}.  \S \ref{ch:TheKochSnowflakePrefractalFlatSurfaceSKSi} contains results on the prefractal flat surface $\sksi{n}$, for any arbitrary $n\geq 0$, and consequences of the fact (also established in \S3, see Theorem \ref{thm:ksnBranchedCoverOfks0}) that such a surface is a branched cover of the hexagonal torus $\sksi{0}$. Most importantly, we show that the billiard flow on the rational polygonal billiard $\omegaksi{n}$ is dynamically equivalent to the geodesic flow on the associated prefractal flat surface $\sksi{n}$.  Such a result allows us to deduce that the set of directions for which a billiard orbit is closed (resp., dense) in $\omegaksi{0}$ is exactly the set of directions for which an orbit is closed (resp., dense) in $\omegaksi{n}$, for any $n\geq 0$. (See Theorem \ref{thm:gutkinsResultAppliedToPrefractals} and Corollary \ref{cor:gutkinsResultAppliedToPrefractals}.)

This fact is used extensively in \S\ref{ch:TheKochSnowflakePrefractalBilliard} when describing orbits of $\omegaksi{n}$, $n\geq 0$, and constructing what we call a  \textit{sequence of compatible orbits}. Such a sequence consists of orbits of each of the billiards $\omegaksi{n}$, for $n=0,1,2,...$, with initial conditions that are themselves compatible in a suitable manner; see Definitions \ref{def:compatibleInitialConditions}, \ref{def:sequenceOfCompatibleInitialConditions} and \ref{def:SequenceOfCompatibleOrbits}.  Using an addressing system developed in \S\ref{subsec:SymbolicRepresentationOfTheTernaryCantorSet}, we define what we are calling \textit{hybrid orbits} of $\omegaksi{n}$; see Definition \ref{def:hybridOrbit}.  In a very concrete sense, hybrid orbits are orbits that survive the construction of $\omegaksi{n+1}$ from $\omegaksi{n}$. We show that dense orbits of $\omegaksi{n}$ are actually dense \textit{hybrid} orbits and we provide sufficient conditions under which a sequence of compatible orbits is a sequence of compatible periodic hybrid orbits. (See Proposition \ref{prop:denseOrbitIsADenseHybridOrbit}, along with Theorems \ref{thm:hybridOrbitOfKS0ImpliesHybridOrbitOfKSn} and \ref{thm:bodd}.) In addition to this, we establish a \textit{topological dichotomy} for sequences of compatible orbits: a sequence of compatible orbits is either entirely comprised of closed orbits or entirely comprised of dense hybrid orbits.  (See Theorem \ref{thm:ATopologicalDichotomy}.)

We would like to suggest that such sequences of compatible orbits should have suitable limits that constitute orbits of the Koch snowflake billiard.  As we will see, there are hybrid orbits of $\omegaksi{n}$ that remain fixed in every subsequent approximation $\omegaksi{N}$, $n\geq N$, for some integer $N\geq 0$; see Theorem \ref{thm:SufficientConditionForCantorOrbit} and Example \ref{exa:AConstantSequenceOfCompatiblePeriodicHybridOrbits}.  Such hybrid orbits have basepoints corresponding to so-called \textit{Cantor-points} (i.e., points of $\ks$ which belong to some finite approximation $\ksi{n}$ but are not vertices of $\ksi{n}$; see \S\ref{subsec:TheKochSnowflake}), and they certainly constitute periodic orbits of the Koch snowflake fractal billiard. In \S\ref{sec:nontrivialPolygonalPaths}, a particular subset of basepoints can be derived from a certain sequence of compatible periodic hybrid orbits which is converging to a point of the snowflake curve called an \textit{elusive limit point} (i.e., a point of $\ks$ which does not belong to any polygonal approximation $\ksi{n}$, for $n\geq 0$;  see \S\ref{subsec:TheKochSnowflake}). (An interesting example of such a situation is provided by a sequence of compatible `hook orbits', discussed in Example \ref{exa:ASequenceOfCompatibleHookOrbits}; see also Example \ref{exa:ASequenceOfCompatiblePeriodicHybridOrbits} further discussed in \S\ref{sec:nontrivialPolygonalPaths}, as well as Conjecture \ref{conj:ConvergingToAnElusiveLimitPoint}.) Such basepoints can then be connected to form what we call a \textit{nontrivial polygonal path} of $\omegaks$.  Furthermore, we consider the concatenation of two suitably chosen nontrivial polygonal paths, thereby connecting two elusive limit points of $\ks$ in a well-defined manner, as is done in \S\ref{sec:nontrivialPolygonalPaths}.




In \S\ref{sec:concludingRemarks}, since the field of ``fractal billiards'' is still in its infancy, we discuss directions for future research and provide several open questions and conjectures regarding the `\textit{fractal flat surface}' $\sks$ and the generalization of our results to other fractal tables.  A number of these open questions are addressed in current works in progress (i.e., \cite{LapNie3,LapNie4}). For a more comprehensive, and somewhat different, list of conjectures, we refer the interested reader to \cite[\S 6]{LapNie2}.


\section{Background and preliminaries}
\label{sec:BackgroundAndPreliminaries}
\subsection{Mathematical Billiards}
\label{subsec:mathematicalBilliards}

Under ideal conditions, we know that a point mass making a perfectly elastic collision with a $C^1$ surface (or curve) will reflect at an angle which is equal to the angle of incidence, this being referred to as the \textit{law of reflection}.

Consider a compact region $\Omega(B)$ in the plane with connected boundary $B$.  Then, $\Omega(B)$ is called a \textit{planar billiard} when $B$ is smooth enough to allow the law of reflection to hold, off of a set of measure zero (where the measure is taken to be the Hausdorff measure or the arc length measure).  Though the law of reflection implicitly states that the angles of incidence and reflection be determined with respect to the normal to the line tangent at the basepoint, we adhere to the equivalent convention in the field of mathematical billiards that the vector describing the position and velocity of the billiard ball (which amounts to the position and angle, since we are assuming unit speed) be reflected in the tangent to the point of incidence.  That is, employing such a law in order to determine the path on which the billiard ball departs after impact essentially amounts to identifying certain vectors.

Then we can rigorously reformulate the law of reflection as follows: the vector describing the direction of motion is the reflection---through the tangent at the point of collision---of the translation of the vector previously describing the direction of motion.  One may express the law of reflection in terms of equivalence classes of vectors by identifying  these two vectors to form an equivalence class of vectors in the unit tangent bundle corresponding to the billiard table $\Omega(B)$ (see Figure \ref{fig:billiardMap}). (See \cite{Sm} for a detailed discussion of this equivalence relation on the unit tangent bundle $\Omega(B)\times S^1$.)

The billiard map $f_B$ is defined on the boundary $B$ of the billiard table.  Really, $f_B:(B\times S^1)/\sim\to (B\times S^1)/\sim$, where the equivalence relation $\sim$ is as discussed above.  More precisely, if $\theta^0$ is an inward pointing vector at a basepoint $x^0$, then $(x^0,\theta^0)$ is the representative element of the equivalence class $[(x^0,\theta^0)]$ and $f_B^k([x^0,\theta^0])=[(x^{k},\theta^{k})]$, where $f_B^k:=f_B\circ...\circ f_B$ is the $k$th iterate of $f_B$.




When $B$ is a nontrivial, connected polygon in $\mathbb{R}^2$, $\Omega(B)$ is called a \textit{polygonal billiard}.  The collection of vertices of $\Omega(B)$ forms a set of zero measure (when we take our measure to be the Hausdorff measure or simply, the arc-length measure on $B$), since there are finitely many vertices. A \textit{rational billiard} is defined below.

\begin{definition}[Rational polygon and rational billiard]
If $B$ is a nontrivial connected polygon such that for each interior angle $\theta_j$ of $B$ there are relatively prime integers $p_j \geq 1$ and $q_j\geq 1$ such that $\theta_j = \frac{p_j}{q_j} \pi$, then we call $B$ a \textit{rational polygon} and $\Omega(B)$ a \textit{rational billiard}.
\label{def:ratBilliard}
\end{definition}

\begin{figure}
\begin{center}
\includegraphics[scale = .55]{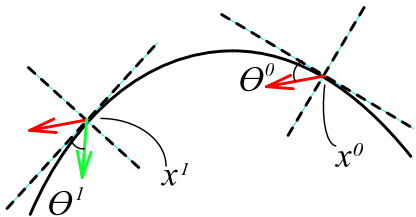}
\end{center}
\caption[Recovering the law of reflection]{A billiard ball traverses the interior of a billiard and collides with the boundary.  The velocity vector is pointed outward at the point of collision. The resulting direction of flow is found by either reflecting the vector through the tangent or by reflecting the incidence vector through the normal and reversing the direction of the vector.  We use the former method in this paper.}
\label{fig:billiardMap}
\end{figure}

\begin{remark}
\label{idx:billiardMapKSn}
In the sequel, we will simply refer to an element $[(x^k,\theta^{k})]$ by $(x^k,\theta^k)$, since the vector corresponding to $\theta^k$ is inward pointing at the basepoint $x^k$.  So as not to introduce unneccessary  notation, when we discuss the billiard map $f_{\ksi{n}}$ corresponding to the $n$th prefractal billiard $\omegaksi{n}$, we will simply write $f_{\ksi{n}}$ as $f_n$.  When discussing the discrete billiard flow on $(\omegaksi{n}\times S^1)/\sim$, the $k$th point in an orbit $(x^k,\theta^k)\in (\omegaksi{n}\times S^1)/\sim$ will instead be denoted by $(\xii{n}{k_n},\theta_n^{k_n})$, so as to be clear as to which space such a point belongs.  Specifically, $k_n$ refers to the number of iterates of the billiard map $f_n$ necessary to produce the pair $(\xii{n}{k_n},\theta_n^{k_n})$. An initial condition of an orbit of $\omegaksi{n}$ will always be referred to as $(\xio{n},\theta_n^0)$.
\end{remark}

In the event that a basepoint $x^j$ of $f_B^j(x^0,\theta^0)$ is a corner of $\Omega(B)$ (that is, a vertex of the polygonal boundary $B$), then the resulting closed orbit is said to be \textit{singular}.  In addition, there exists a positive integer $k$ such that the basepoint $x^{-k}$ of $f_B^{-k}(x^0,\theta^0)$ is a corner of $\Omega(B)$. (Here, $f_B^{-k}$ denotes the $k$th inverse iterate of $f_B$.)  The path then traced out by the billiard ball connecting $x^j$ and $x^{-k}$ is called a \textit{saddle connection}.

\begin{definition}[Footprint of an orbit]
\label{def:footprintOfAnOrbit}
Let $\orbitksi{n}$ be an orbit of $\omegaksi{n}$.  Then
\begin{align}
\footprintksi{n} &:= \orbitksi{n}\cap\ksi{n}
\end{align}
\noindent is called the \textit{footprint} of the orbit $\omegaksi{n}$.
\end{definition}


\subsection{Flat surfaces and properties of the flow}
\label{subsec:flatStructuresandFlatSurfaces}
In this section, we deal only with flat surfaces constructed from rational billiards.

\begin{definition}[Flat structure and flat surface]
\label{def:flatStructure}
Let $M$ be a compact, connected, orientable surface.  A \textit{flat structure} on $M$ is an atlas $\omega$, consisting of charts of the form $(U_\alpha,\varphi_\alpha)_{\alpha\in\mathscr{A}}$, where $U_\alpha$ is a domain (i.e., a connected open set) in $M$ and $\varphi_\alpha$ is a homeomorphism from $U_\alpha$ to a domain in $\mathbb{R}^2$, such that the following conditions hold:

\begin{enumerate}
\item{the collection $\{U_\alpha\}_{\alpha\in\mathscr{A}}$ cover the whole surface $M$ except for finitely many points $z_1,z_2,...,z_k$, called \textit{singular points};}
\item{all coordinate changing functions are translations in $\mathbb{R}^2$;}
\item{the atlas $\omega$ is maximal with respect to properties $(1)$ and $(2)$;}
\item{for each singular point $z_j$, there is a positive integer $m_j$, a punctured neighborhood $\dot{U}_j$ of $z_j$ not containing other singular points, and a map $\psi_j$ from this neighborhood to a punctured neighborhood $\dot{V}_j$ of a point in $\mathbb{R}^2$ that is a shift in the local coordinates from $\omega$, and is such that each point in $\dot{V}_j$ has exactly $m_j$ preimages under $\psi_j$.}
\end{enumerate}

\noindent We say that a connected, compact surface equipped with a flat structure is a flat surface.
\end{definition}

\begin{remark}
Note that in the literature on billiards and dynamical systems, the terminology and definitions pertaining to this topic are not completely uniform; see, for example, \cite{GaStVo,Gut1,GutJu1,GutJu2,HaKa,HuSc,Mas,MasTa,Ve1,Ve2,Vo,Zo}.  We have adopted the above definition for clarity and the reader's convenience.
\end{remark}

We now discuss how to construct a flat surface from a rational billiard.  Consider a rational polygon billiard $\Omega(P)$ with $k$ sides and interior angles $\frac{p_j}{q_j}\pi$ at each vertex $z_j$, for $1\leq j\leq k$, where the positive integers $p_j$ and $q_j$ are relatively prime.  The linear portions of the planar symmetries generated by reflection in the sides of the polygonal billiard $\Omega(P)$ generate the dihedral group $D_N$, where $N:=\text{lcm} \{q_j\}_{j=1}^k$. Next, we consider $\Omega(P)\times D_N$ (equipped with the product topology).  We want to glue `sides' of $\Omega(P)\times  D_N$ together and construct a natural atlas on the resulting surface $M$ so that $M$ becomes a flat surface.






As a result of the identification, the points of $M$ that correspond to the vertices of $\Omega(P)$ constitute (removable or  nonremovable) conic singularities of this surface.  Heuristically, $\Omega(P)\times D_N$ can be represented as $\{r_j\Omega(P)\}_{j=1}^{2N}$, in which case it is easy to see what sides are made equivalent under the action of $\sim$.  That is, $\sim$ identifies opposite and parallel sides in a manner which preserves the orientation.


\subsection{Unfolding a billiard orbit}
\label{subsec:UnfoldingABilliardOrbit}
Consider a rational polygonal billiard $\Omega(B)$ and an orbit $\mathscr{O}(x^0,\theta^0)$.  Reflecting the billiard $\Omega(B)$ and the orbit in a side of the billiard containing a basepoint of the orbit (or an element of the footprint of the orbit) partially unfolds the orbit $\mathscr{O}(x^0,\theta^0)$; see Figure \ref{fig:UnfoldingAnOrbitOfTheEquilateralTriangleBilliard}. Continuing this process until the orbit is a straight line produces as many copies of the billiard table as there are elements of the footprint. That is, if the period of an orbit $\mathscr{O}(x^0,\theta^0)$ is some positive integer $p$, then the number of copies of the billiard table in the unfolding is also $p$. Therefore, we refer to such a straight line as the unfolding of the billiard orbit.

\begin{figure}
\begin{center}
\includegraphics[scale=.35]{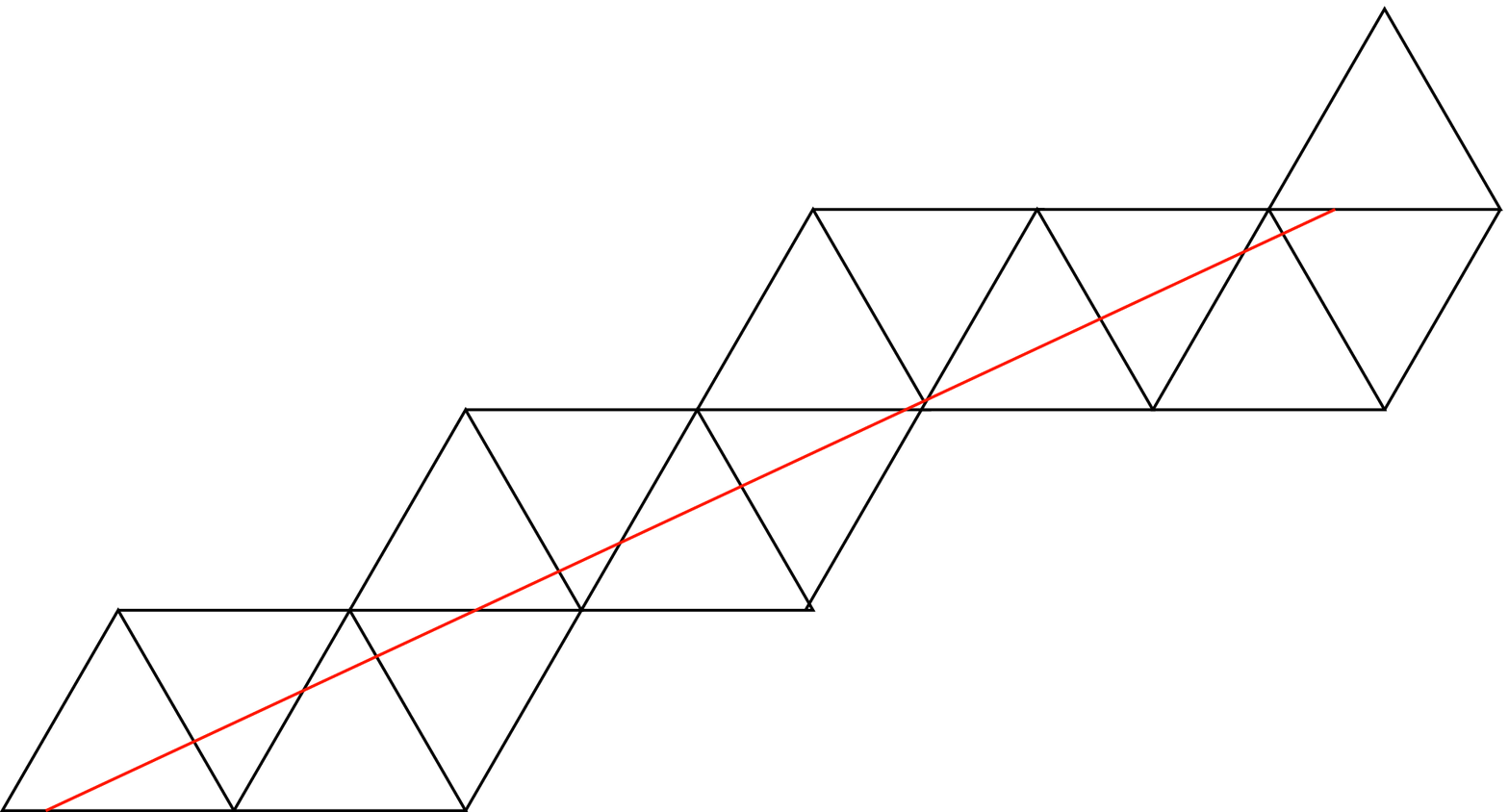}
\end{center}
\caption{Unfolding an orbit of the equilateral triangle billiard $\omegaksi{0}$.}
\label{fig:UnfoldingAnOrbitOfTheEquilateralTriangleBilliard}
\end{figure}

\subsection{Symbolic representation of the ternary Cantor set.}
\label{subsec:SymbolicRepresentationOfTheTernaryCantorSet}

Every point of the ternary Cantor set $\mathscr{C}$ (hereafter referred to as the Cantor set) has a ternary \textit{representation} given in terms of the characters $l$, $c$ and $r$ (standing for left, center and right, respectively).  For example, in terms of our alphabet, we say $1/3$ has a ternary representation given by $l\overline{r}$ and $1/4$ has a ternary representation given by $\overline{lr}$, where the over-bar indicates that the corresponding string of symbols is repeated ad infinitum.  We stress that a ternary \textit{representation} will always consist of infinitely many characters, while a ternary \textit{expansion} of an element of the unit interval $I$ can be finite, which is illustrated by the example of the rational value $1/3=0.1$ in base-$3$.  For the sake of simplicity, we take every element of $\cantor$ to have a ternary representation that contains no $c$'s. (What we want to prevent is an element of $\cantor$ having a representation determined by approaching it from the complement of $\cantor$ in $I=[0,1]$.  In this way, every point of $\cantor$ has a unique ternary representation.)

In the sequel, the \textit{type} of ternary representation will provide us with important information.  Particular qualities of the representation will, in part, dictate the type of resulting orbit and the nature of the sequence of compatible orbits.


\begin{notation}
\label{not:typeOfTernaryRepresentation}
The \textit{type of ternary representation} can be defined as follows.  If $x\in I$, then the first coordinate of $[\cdot, \cdot]$ describes the characters that occur infinitely often and the second coordinate of $[\cdot,\cdot]$ describes the characters that occur finitely often.  If we want to discuss many different types of ternary representations, then we use `or.'  That is, the notation $[\cdot,\cdot]\vee [\cdot,\cdot]\vee...\vee[\cdot,\cdot]$ is to be read as \textit{$[\cdot,\cdot]$ or $[\cdot,\cdot]$ or ... or $[\cdot,\cdot]$}. If the collection of characters occurring finitely often is empty, then we denote the corresponding type of ternary representation by $\tern{\cdot}{\emptyset}$.
\end{notation}

\begin{example}
If an element $x\in I=[0,1]$ has a ternary representation consisting of infinitely many $c$'s and $l$'s but finitely many $r$'s, then we write this type of ternary representation as $[lc,r]$.  If we have a collection of points in $I$ such that each point has a ternary representation consisting of infinitely many $c$'s and $l$'s and finitely many $r$'s or else infinitely many $l$'s and $r$'s and finitely many $c$'s, then we write the corresponding types of ternary representations as $[lc,r]\vee[lr,c]$.
\end{example}

\subsection{The Koch snowflake}
\label{subsec:TheKochSnowflake}
The Koch snowflake $\ks$ is a compact, connected and infinitely long curve in the plane with the property that at no point of the snowflake $\ks$ can one form a well-defined tangent.  This last property is what makes defining a law of reflection on the Koch snowflake billiard boundary so difficult.


The Koch snowflake is constructed, as shown in Figure \ref{fig:kochconstruction75intro}, by removing the open middle third of each successive side of length $1/3^{n-1}$ and placing at each pair of endpoints two uprights that would have formed the sides of an equilateral triangle with side lengths measuring $1/3^n$.

Next, we define what a \textit{cell} of the Koch snowflake billiard is.

\begin{definition}[A cell $C_{n,\nu}$ of $\omegaksi{n}$]
\label{def:cellOfOmegaKSn}
Consider (the `set-theoretic difference') $\omegaksi{n}\setminus \omegaksi{n-1}$.  Each resulting triangular region is then called a \textit{cell} of $\omegaksi{n}$. We denote a cell of $\omegaksi{n}$ by $C_{n,\nu}$, where $\nu$ denotes the side of $\omegaksi{n-1}$ to which the cell was glued; see Figure \ref{fig:illustrateDefiOfCell}.  Hence, there are $3\cdot 4^{n-1}$ cells $C_{n,\nu}$ of $\omegaksi{n}$ and so $1\leq \nu\leq 3\cdot 4^{n-1}$.
\end{definition}

\begin{figure}
\begin{center}
\includegraphics[scale=1.25]{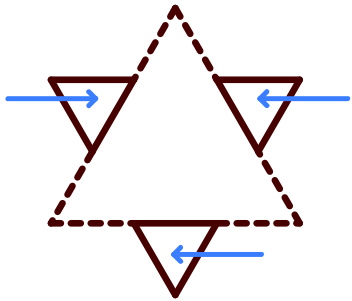}
\end{center}
\caption[An illustration of a cell $C_{1,\nu}$]{An illustration of Definition \ref{def:cellOfOmegaKSn} in terms of $\ksi{0} = \Delta$ and $\ksi{1}$.  Represented are the three cells $C_{1,1}$, $C_{1,2}$ and $C_{1,3}$ of $\omegaksi{1}$.}
\label{fig:illustrateDefiOfCell}
\end{figure}

In \S\ref{ch:TheKochSnowflakePrefractalBilliard}, we will be interested in the information provided by the ternary representation of an element $\xio{n}$ of a side $s_{n,\nu}$ of $\omegaksi{n}$.  We have already seen how to represent elements of the unit interval $I$. An element of a side $s_{n,\nu}$ of $\omegaksi{n}$ has a ternary representation also given in terms of the characters $l$, $c$ and $r$.






In the Koch snowflake $\ks$, there are three types of points: \textit{corners}, \textit{Cantor-points} and \textit{elusive limit points}.  A \textit{corner} of the Koch snowflake $\ks$ is a point of $\ks$ that is a corner of a finite approximation $\ksi{n}$, for some $n\geq 0$.  A \textit{Cantor-point} of $\ks$ is a point of a finite approximation $\ksi{n}$, for some $n\geq 0$, such that the ternary representation of this point (with respect to the side on which it lies) has the form $\tern{lr}{\emptyset}$  (that is, consists of infinitely many $l$'s \textit{and} $r$'s and no $c$'s).  Therefore, a Cantor-point is a point of $\ks$ that is not a corner, but definitely a point of $\ks$ that exists in some finite approximation.  An \textit{elusive limit point} of $\ks$ is a point of $\ks$ that never belongs to any finite approximation $\ksi{n}$.  We will see that it is the Cantor-points and elusive limit points that will be of the greatest interest in the sequel.






\section{The Koch snowflake prefractal flat surface $\sksi{n}$}
\label{ch:TheKochSnowflakePrefractalFlatSurfaceSKSi}
We denote by $\mathcal{S}(P)$ the flat surface $M$ constructed from a particular rational billiard $\Omega(P)$.

In particular, $\mathcal{S}(KS_n)$ is the flat surface associated with the prefractal billiard $\Omega(KS_n)$.  The flat surfaces $\mathcal{S}(KS_n)$, $n=1,2,3$, are given in Figure \ref{fig:kochsurface}.  For each billiard $\Omega(KS_n)$, the group of symmetries $D_N$, where $N=\text{lcm}\{q_j\}_{j=1}^{3\cdot 4^n}$ (that is, the second component in the product $\Omega(KS_n)\times D_N$) is the dihedral group $D_3$, and thus is independent of $n$.  From this, we deduce that for any $n\geq 0$, there are six copies of the prefractal billiard table $\omegaksi{n}$ (with sides appropriately identified) used in the construction of the associated flat surface $\mathcal{S}(KS_n):=(\omegaksi{n}\times D_3)/\sim$; see Figure \ref{fig:kochsurface}.  We refer the reader back to \S\ref{subsec:flatStructuresandFlatSurfaces} for a general discussion of flat surfaces and, e.g., to \cite{Ma} for the topological notions (such as covering map, branched cover) used in this section (esp., in Theorem \ref{thm:ksnBranchedCoverOfks0} and its proof).

\begin{figure}
\begin{center}
\includegraphics[scale = .35]{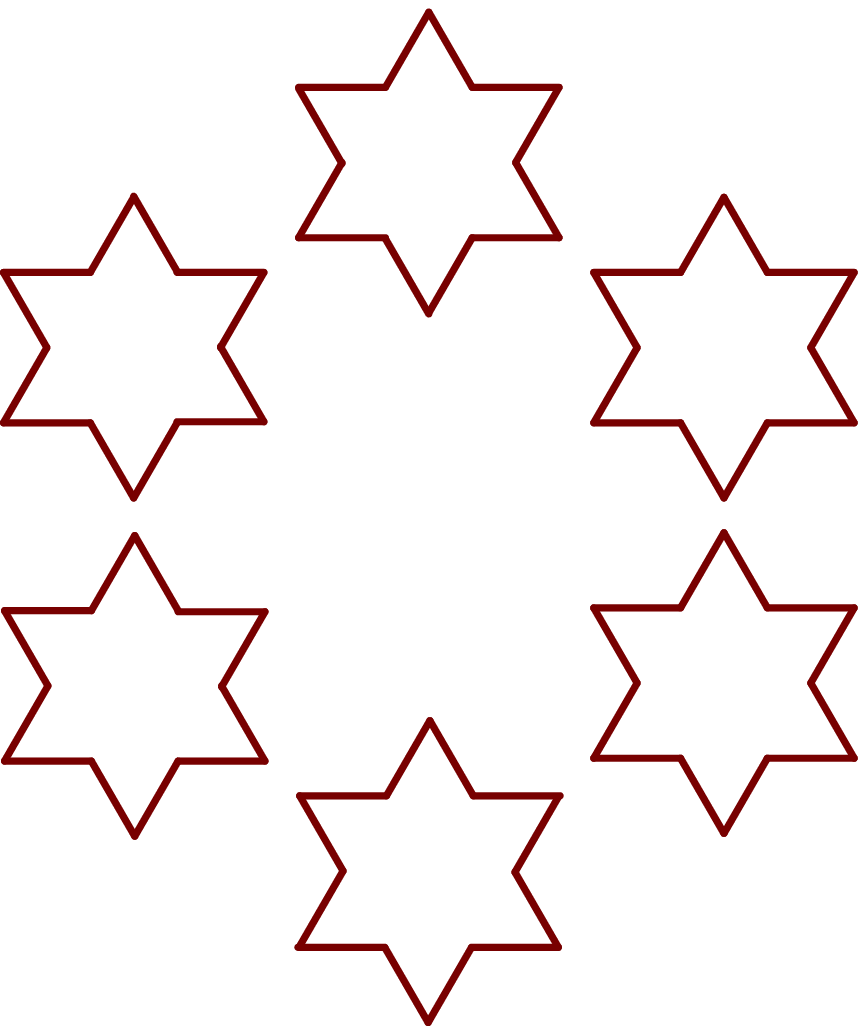}  \quad \includegraphics[scale = .35]{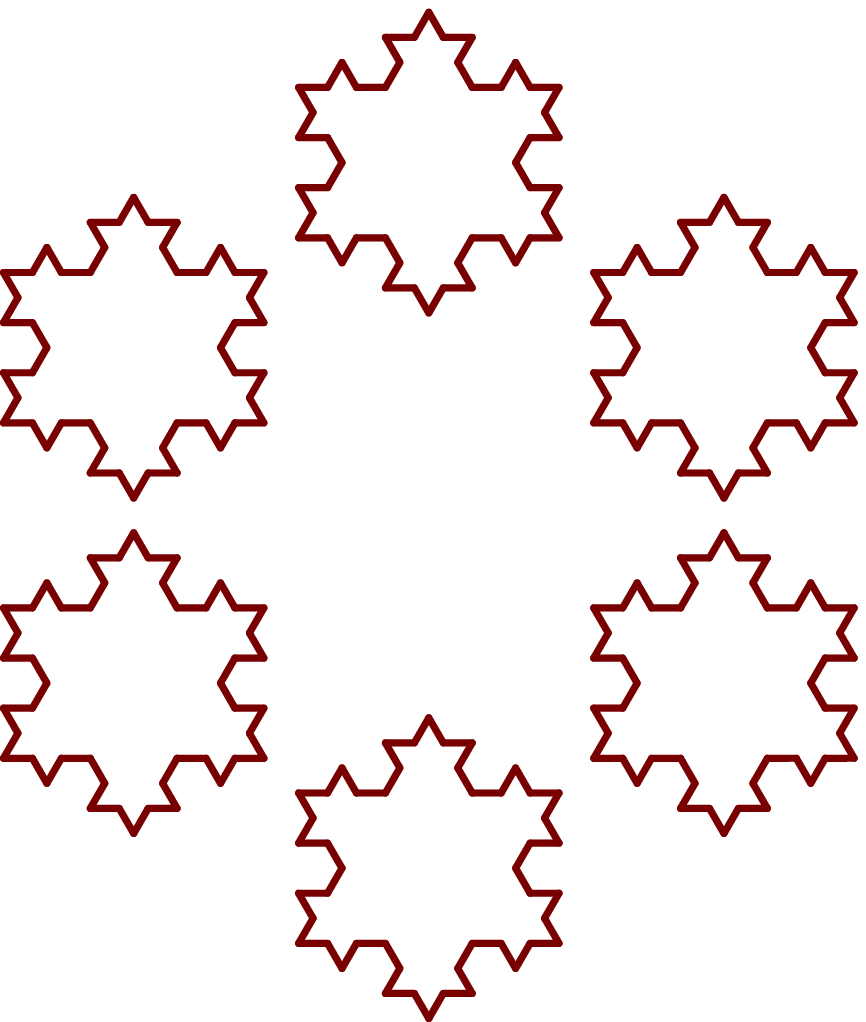} \quad \includegraphics[scale = .35]{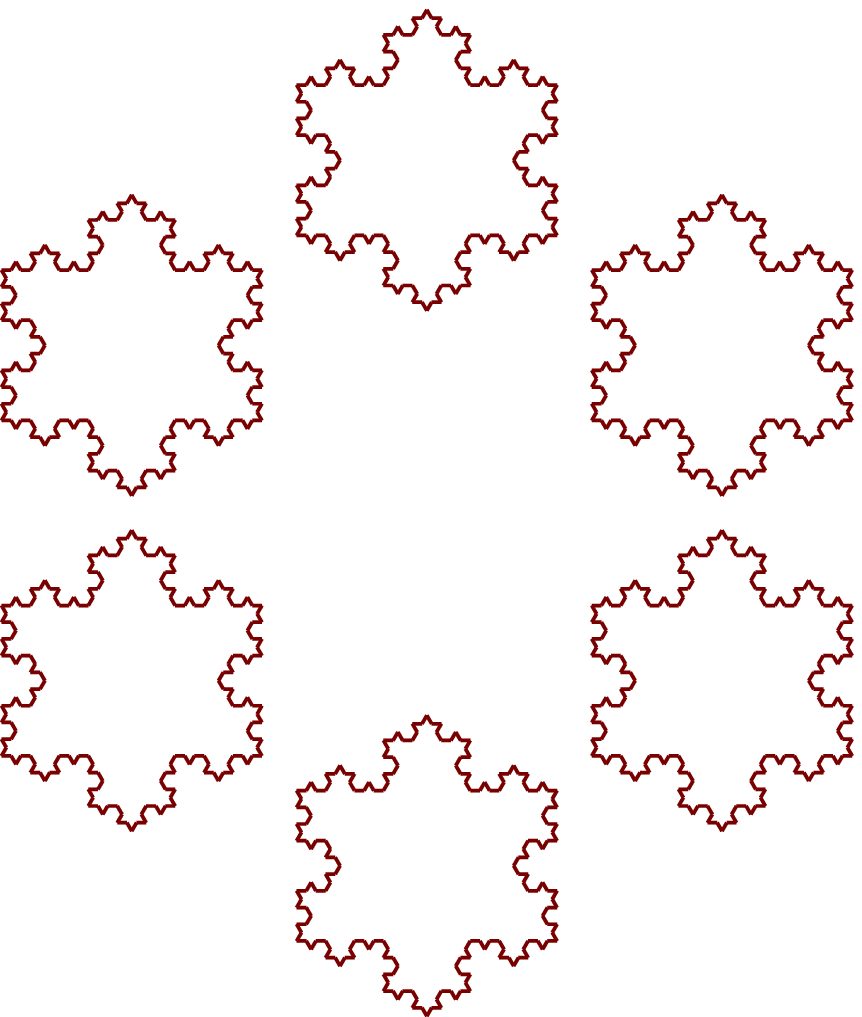}
\end{center}
\caption[Examples of the prefractal flat surface]{The flat surfaces $\mathcal{S}(KS_n)$, for $n=1,2,3$.  Note that the proper identification is not shown in the figures above.  Given the arrangement of the six copies of $KS_n$, one then identifies opposite and parallel sides to make the proper identification that results in a geodesic flow that is dynamically equivalent with the billiard flow on the associated billiards $\omegaksi{1},\omegaksi{2},\omegaksi{3}$.}
\label{fig:kochsurface}
\end{figure}

\begin{remark}
For the remainder of the paper, when we say that a regular polygon is of \textit{scale $n$}, we mean that the side length of the regular polygon is $1/3^n$.  For example, an equilateral triangle of scale $n$ is one for which the side length is $1/3^n$.
\end{remark}


The flat surface $\sksi{n}$ is a surface with both types of conic singularities: removable and nonremovable.  In constructing the flat surface $\sksi{n}$ via $\omegaksi{n}$, we see that the nonremovable conic singularities correspond to corners with obtuse angles of $\omegaksi{n}$ and removable singularities correspond to corners with acute angles (both measured relative to the interior of $\omegaksi{n}$).  Since for every $n\geq 1$, the measure of every obtuse corner is the same (specifically, $4\pi/3$ radians), it follows that the conic angle of a nonremovable singularity is $8\pi$.  For the same reason, every corner with an acute angle (with every acute angle measuring $\pi/6$ radians) gives rise to a removable singularity with conic angle $2\pi$; this is, in fact, a defining characteristic of a removable singularity.

\begin{proposition}
For any $n\geq 0$, the genus $g_n$ of the surface $\sksi{n}$ is given by
\begin{align}
\label{eqn:genus-n}g_n &= 3\cdot 4^n - 2.
\end{align}
\end{proposition}

\begin{proof}
Let $\Omega(P)$ be a rational billiard and let $\{V_j\}_{j=1}^\nu$ denote the $\nu$ vertices of the polygon $P$.  For each $j=1,...,\nu$, let the measure of the angle formed by the vertex $V_j$ be $(p_j/q_j)\pi$, and let $N:=\text{lcm}\{q_j\}_{j=1}^\nu$.  Then, if $g$ is the genus of the corresponding flat surface $\mathcal{S}(P)$, the Euler characteristic $\chi=2-2g$ of that same surface is given by
\begin{align}
\chi &= N\sum_{j=1}^{\nu} \frac{1}{q_j} -N\nu + 2N;\label{eqn:eulerCharacteristic}
\end{align}
see \cite{HuSc} for a detailed description of how to calculate the genus of a surface that arises from a rational billiard table.

The prefractal billiard $\omegaksi{n}$ has $3\cdot 4^n$ many sides and as many vertices.  Moroever, $N_n=\text{lcm}\{q_j\}_{j=1}^{3\cdot 4^n}=3$, for every $n\geq 0$. The Euler characteristic $\chi_n=2-2g_n$ of $\sksi{n}$ is given by Equation (\ref{eqn:eulerCharacteristic}). Therefore, solving for $g_n$, we see that the genus of the prefractal flat surface $\sksi{n}$ is given by Equation (\ref{eqn:genus-n}).
\end{proof}

\subsection{$\sksi{n}$ is a branched cover of $\sksi{0}$}
\label{sec:SKSnAsABranchedCover}

Taking as inspiration the results and methods of Gutkin and Judge in \cite{GutJu1} and \cite{GutJu2}, we now show that for each $n\geq 1$, the flat surface $\mathcal{S}(KS_n)$ is a branched cover of the hexagonal torus $\mathcal{S}(KS_0)$; see Figure \ref{fig:hexagonalTorus}. To such end, we establish several results culminating in this fact.

\begin{figure}
\begin{center}
\includegraphics[scale = .45]{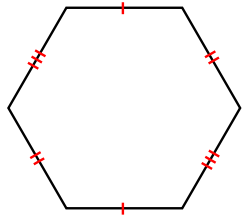}
\end{center}
\caption[The hexagonal torus $\mathcal{S}(\ksi{0})$]{The hexagonal torus $\mathcal{S}(\ksi{0})$. As usual, similarly marked sides are identified.  It should be noted that $\mathcal{S}(KS_0)$ is topologically (but not metrically) equivalent to the flat square torus.}
\label{fig:hexagonalTorus}
\end{figure}

\begin{lemma}
Let $n\in \mathbb{N}$.  Then, for any positive integer $k\geq n$, $\mathcal{S}(KS_n)$ can be tiled by equilateral triangles of scale $k$.
\label{lem:kochTiledByEqui}
\end{lemma}

\begin{proof}
This follows from the construction of the Koch snowflake.  We note that each triangle of scale $n$, denoted $\Delta_n$, can be tiled by $9^{k-n}$ triangles of scale $k\geq n$; see Figure \ref{fig:equiTiledHex} for the case when $k=n+1$.  Note that $\mathcal{S}(KS_n)=(\Omega(KS_n)\times D_3)/\sim$ and that $\Omega(KS_n)$ is constructed from $\Omega(KS_{n-1})$ by gluing a copy of $\Delta_n$ to every side $s_{n-1,\nu}$ at the middle third of $s_{n-1,\nu}$.  Since every triangle $\Delta_{n-1}$ tiling $\omegaksi{n-1}$ can also be tiled by $\Delta_n$, it follows that $\omegaksi{n}$ is tiled by $\Delta_n$.. So, $\sksi{n}$ can be tiled by equilateral triangles of scale $k$, for every $k\geq n$.
\end{proof}

In the sequel, given a bounded set $A\subseteq \mathbb{R}^2$, we will write that ``$A$ can be tiled by $H_n$'' in order to indicate that $A$ can be tiled by finitely many copies of hexagonal tiles $H_n$ of scale $n$.

\begin{figure}
\begin{center}
\includegraphics[scale = .5]{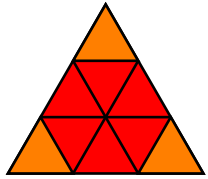}
\end{center}
\caption[A triangle is tiled by smaller triangles]{We see that $\Delta_n$ is tiled by nine copies of $\Delta_{n+1}$, six of which form a hexagonal tile $H_{n+1}$ in the center.}
\label{fig:equiTiledHex}
\end{figure}

\begin{proposition}
\label{prop:hextile}
Let $n\in \mathbb{N}$.  Then the flat surface $\mathcal{S}(KS_n)$ can be tiled by $H_k$, for all $k\geq n+1$, in such a way that each conic singularity is at the center of some tile $H_k$.
\end{proposition}

\begin{proof}
We see in Figure \ref{fig:ks1TiledByH2} that $\sksi{1}$ can be tiled by $H_2$ such that each conic singularity is at the center of some tile $H_2$.  Each $H_2$ is tiled by six equilateral triangles $\Delta_2$.  As was seen in the proof of Lemma \ref{lem:kochTiledByEqui}, each $\Delta_2$ is tiled by nine $\Delta_3$ such that six of these triangles form a hexagonal tile $H_3$; see Figure \ref{fig:equiTiledHex}.  At the center of $H_2$ is a copy of $H_3$.  Hence, each conic singularity remains at the center of some tile $H_3$; see Figure \ref{fig:hexagonTiled}.  Continuing in this fashion, we see that for each $k\geq 2$, $H_k$ tiles $\sksi{2}$ in such a way that each conic singularity is at the center of some $H_k$.

Suppose there exists $N\in \mathbb{N}$ such that, for every $n\leq N$, $\sksi{n}$ can be tiled by $H_k$, for every $k\geq n+1$.  In particular, $\sksi{N}$ can be tiled by $H_k$, for every $k\geq N+1$.  We then have that, for every $k\geq N+2$, $\sksi{N}$ can be tiled by $H_k$.  By Lemma \ref{lem:kochTiledByEqui}, $\Delta_{N+1}$ tiles $\sksi{N+1}$.  Each triangular region $\Delta_{N+1}$ in $\sksi{N+1}$ but not in $\sksi{N}$ is tiled by nine triangles $\Delta_{N+2}$ in such a way that six $\Delta_{N+2}$ comprise a tile $H_{N+2}$.  Continuing in this fashion, we see that each $\Delta_k$ contributes to a hexagonal tile $H_k$ (as part of the embedded tiling) in such a way that each conic singularity is at the center of some hexagonal tile $H_k$.
\end{proof}

\begin{figure}
\begin{center}
\includegraphics[scale=.55]{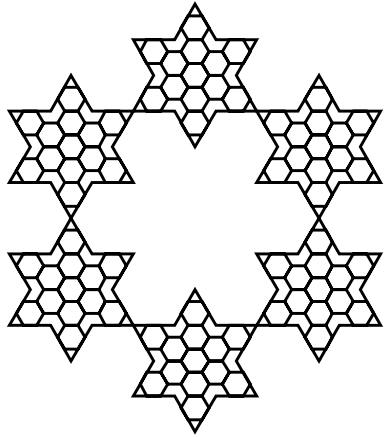}
\end{center}
\caption[Tiling $\sksi{1}$ by $H_2$]{Tiling the flat surface $\sksi{1}$ by hexagonal tiles $H_2$.  We note that the conic singularities (both removable and nonremovable) are at the center of hexagonal tiles.}
\label{fig:ks1TiledByH2}
\end{figure}

\begin{figure}
\begin{center}
\includegraphics[scale = .4]{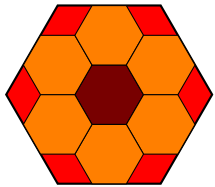}
\end{center}
\caption[Six triangles $\Delta_n$ tile $H_n$]{Six triangles $\Delta_n$ tile $H_n$.  The hexagonal tile $H_n$ is tiled by seven tiles $H_{n+1}$ with six rhombic tiles.}
\label{fig:hexagonTiled}
\end{figure}


\begin{theorem}
\label{thm:ksnBranchedCoverOfks0}
For every $n\in \mathbb{N}$, the prefractal Koch snowflake flat surface $\mathcal{S}(KS_n)$ is a branched cover of the \emph{(}singly punctured\emph{)} prefractal Koch snowflake flat surface $\mathcal{S}(KS_0)$, which is the hexagonal torus.  Such a covering map $p_n:\sksi{n}\to\sksi{0}$ is given by suitably defined translations on $\sksi{n}$.
\end{theorem}

\begin{proof}
The center point $x_0$ of the flat hexagonal torus $\mathcal{S}(KS_0)$ is a branched locus of the cover $\mathcal{S}(KS_n)$ when $\mathcal{S}(KS_n)$ is tiled by $H_{n+1}$ as described in Proposition \ref{prop:hextile}.  This follows from the fact that every nonremovable conic singularity of $\sksi{n}$ is at the center of four hexagonal tiles.  Specifically, this means that this center point $x_0$ is not evenly covered by the covering map $p_n:\mathcal{S}(KS_n)\rightarrow \sksi{0}$ determined by suitable translations of hexagonal tiles $H_{n+1}$ on $\sksi{n}$.  Any other point in $\mathcal{S}(KS_0)$ is evenly covered since every element in the fiber $p_n^{-1}(z)$, $z\neq x_0$, has a conic angle of $2\pi$.
\end{proof}


\subsection{Minimality of the flow on $\omegaksi{n}$ and its consequences.}
\label{subsec:minimalityOfTheFlowOnOmegaKSnAndConsequences}







When discussing billiard orbits of $\omegaksi{n}$, we will find it more convenient to measure angles of incidence and reflection relative to a fixed coordinate system.  As such, we suppose that the left corner of the equilateral triangle with side length one constitutes the origin.  However, on occasion, we will find it useful to refer to the angle of reflection measured relative to a particular side.  So that no confusion arises, when we are discussing such a situation, if  $\varpi$ is an angle measured relative to a side of $\omegaksi{n}$, then $\theta(\varpi)$ is the same angle measured relative to the fixed coordinate system.

If $\{u_1,u_2\}$ is a basis for $\mathbb{R}^2$, then a vector $z\in \mathbb{R}^2$  is called \textit{rational with respect to} $\{u_1,u_2\}$ if $z=mu_1+nu_2$, for some $m,n\in \mathbb{Z}$. Combining the results of \cite{GutJu1} with Theorem 3 of \cite{Gut2}, we can state the following result, which we do not claim as a new theorem, but which we rephrase in a way that is suitable for our purposes.

\begin{theorem}[\cite{Gut2}]
\label{thm:gutkinsResult}
Let $\mathcal{S}(P)$ be a flat surface determined from a rational polygonal billiard $\Omega(P)$.  If $\mathcal{S}(P)$ is a branched cover of a singly punctured torus, then a geodesic on $\mathcal{S}(P)$ is periodic or forms a saddle connection if and only if the geodesic has an initial direction that is rational.  In addition, a geodesic on $\mathcal{S}(P)$ is dense if and only if the geodesic has an initial direction that is irrational.
\end{theorem}

\begin{figure}
\begin{center}
\includegraphics[scale=0.35]{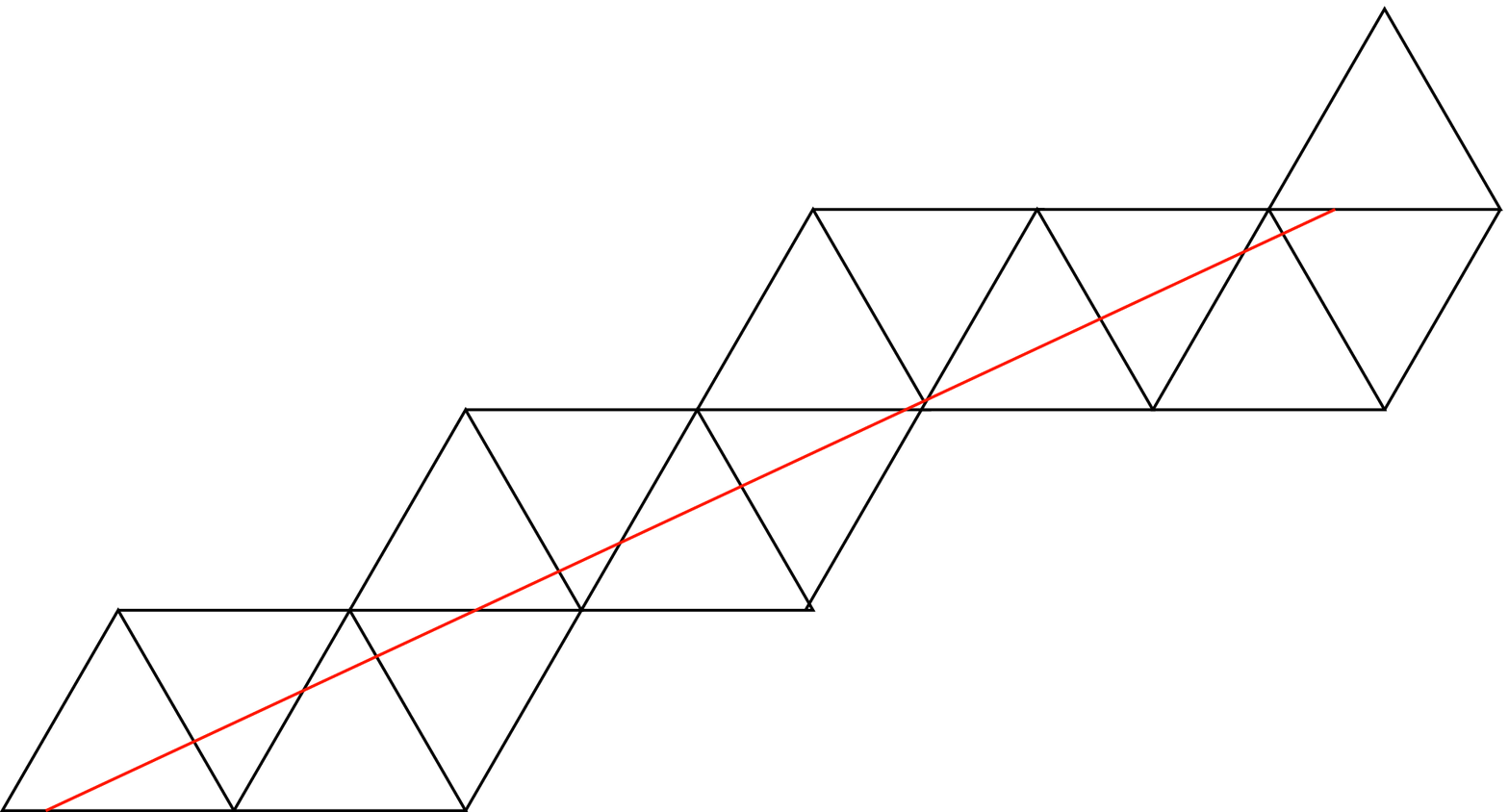}
\caption{The lattice points constitute linear integer combinations of the basis vectors $\{u_1,u_2\}= \{(1,0),(1/2,\sqrt{3}/2)\}$. Here we show an unfolded orbit to emphasize the utility of such a tool.  The unfolded orbit has an initial direction that is rational, meaning such an orbit will be closed in the equilateral triangle.}
\label{fig:unfoldingAnOrbit}
\end{center}
\end{figure}

By what we saw in \S\S\ref{subsec:flatStructuresandFlatSurfaces} and \ref{subsec:UnfoldingABilliardOrbit}, the geodesic flow on $\sksi{n}$ is dynamically equivalent to the billiard flow on $\omegaksi{n}$. In \S\ref{sec:SKSnAsABranchedCover}, we proved that $\mathcal{S}(KS_n)$ is a branched cover of the singly punctured torus $\mathcal{S}(KS_0)$.  Applying Theorem \ref{thm:gutkinsResult} to the Koch snowflake prefractal flat surfaces, we then obtain the following result.

\begin{theorem}
\label{thm:gutkinsResultAppliedToPrefractals}
Let $n\geq 0$.  A direction $\theta$ is a rational direction \emph{(}with respect to the basis $\{u_1,u_2\}= \{(1,0),(1/2,\sqrt{3}/2)\}$\emph{)} if and only if a geodesic in the direction of $\theta$ on $\sksi{n}$ is periodic or forms a saddle connection. Furthermore, a direction $\theta$ is an irrational direction \emph{(}with respect to $\{u_1,u_2\}$\emph{)} if and only if a geodesic in the direction of $\theta$ on $\sksi{n}$ is dense.
\end{theorem}

Because the geodesic flow on the prefractal flat surface $\sksi{n}$ is dynamically equivalent to the billiard flow on the corresponding billiard table $\omegaksi{n}$, we can state Theorem \ref{thm:gutkinsResultAppliedToPrefractals} in terms of the billiard flow on the prefractal billiard $\omegaksi{n}$.

\begin{corollary}
\label{cor:gutkinsResultAppliedToPrefractals}
Let $n\geq 0$.  A direction $\theta$ is a rational direction \emph{(}with respect to $\{u_1,u_2\}= \{(1,0),(1/2,\sqrt{3}/2)\}$\emph{)} if and only if an orbit in the direction of $\theta$ of $\omegaksi{n}$ is closed.  Furthermore, a direction $\theta$ is an irrational direction \emph{(}with respect to $\{u_1,u_2\}$\emph{)} if and only if an orbit with the initial direction of $\theta$ in $\omegaksi{n}$ is dense.
\end{corollary}

\section{Hybrid orbits of the Koch snowflake prefractal billiard $\omegaksi{n}$}
\label{ch:TheKochSnowflakePrefractalBilliard}


We want to construct sequences of orbits in such a way that one orbit is suitably related to another. More precisely, we want to develop a notion of `compatibility' that relates an orbit of $\omegaksi{n}$ to an orbit of $\omegaksi{n+1}$, for each $n\geq 0$.  Let us first consider an orbit of $\omegaksi{0}$.  Such an orbit has basepoints that potentially lie on segments that are removed as part of the construction of finer prefractal approximations, on the Cantor set that remains as part of the construction process or both.  Obviously, an orbit of $\omegaksi{0}$ may not be an orbit of $\omegaksi{1}$.  However, depending on the types of the ternary representations of the basepoints of the orbit of $\omegaksi{0}$, a sequence of compatible orbits will exhibit particularly interesting dynamical behavior. While certain orbits in a so-called sequence of compatible orbits will form saddle connections in their respective billiard tables, we will see that this is the exception rather than the rule.

A hybrid orbit of a prefractal billiard is an orbit of $\omegaksi{n}$ for which the ternary representation of the elements of the corresponding footprint are such that they never correspond to points of $\ks$ with finite ternary representations (i.e., corners). As we will see, certain orbits remain constant from one prefractal billiard $\omegaksi{n}$ to the next and certain orbits change entirely with each prefractal billiard. The term \textit{hybrid} is meant to indicate that such orbits have qualities that are found in these two types of orbits mentioned in the previous sentence.  As such, Definition \ref{def:hybridOrbit} is phrased so as to include these two types of orbits and more general orbits that have qualities reminiscent of both.


\begin{definition}[Hybrid orbit]
\label{def:hybridOrbit}
Let $\orbitksi{n}$ be an orbit of $\omegaksi{n}$.  If all but at most two basepoints  $\xii{n}{k_n}\in \footprintksi{n}$ have ternary representations (determined with respect to the side $s_{n,\nu}$ on which each point resides) of type $\tern{c}{lr}\vee\tern{cl}{r}\vee\tern{cr}{l}\vee\tern{lcr}{\emptyset}\vee\tern{lr}{\emptyset}$, then we call $\orbitksiang{n}{\theta_n^0}$ a \textit{hybrid orbit} of $\omegaksi{n}$.
\end{definition}

\begin{definition}[A $\mathscr{P}$ hybrid orbit]
If $\orbitksi{n}$ is a hybrid orbit with property $\mathscr{P}$, then we say that it is a $\mathscr{P}$ hybrid orbit.
\end{definition}




\begin{proposition}
\label{prop:denseOrbitIsADenseHybridOrbit}
If $\orbitksi{n}$ is a dense orbit of $\omegaksi{n}$, then $\orbitksi{n}$ is a dense hybrid orbit.
\end{proposition}

\begin{proof}
Suppose there were two basepoints $\xii{n}{k_n}$ and $\xii{n}{k'_n}$ of a dense orbit $\orbitksi{n}$ with ternary representations of types $\tern{l}{cr}\vee\tern{r}{lc}$.  Then, there exists $N\geq n$ such that the orbit connects two vertices of two equilateral triangles of scale $N$ tiling $\omegaksi{n}$.  Since this orbit can be unfolded (much as in \S\ref{subsec:UnfoldingABilliardOrbit}) into the corresponding flat surface and then projected down onto the hexagonal torus, such an orbit (or flow line on the flat surface) must be at least a saddle connection of the equilateral triangle billiard.  However, such a direction $\theta_n^0$ should yield a dense billiard flow in $\omegaksi{0}$, which is not the case.  Hence, $\xii{n}{k_n}$ and $\xii{n}{k'_n}$ do not both have a ternary representation of type $\tern{l}{cr}\vee\tern{r}{lc}$.  Moreover, if any basepoint of $\orbitksi{n}$ already corresponds to a corner of $\omegaksi{n}$, then a similar argument shows that no other basepoint may have a ternary representation of type $\tern{l}{cr}\vee\tern{r}{lc}$.  Therefore, $\orbitksi{n}$ is a dense hybrid orbit.
\end{proof}




We will construct sequences of suitably related orbits with the intention of examining the limiting behavior of such sequences.  This will be the focus of the latter part of this section and of \S\ref{sec:nontrivialPolygonalPaths}.  To such end, we define a \textit{sequence of compatible initial conditions} below.

\begin{definition}[Compatible initial conditions]
\label{def:compatibleInitialConditions}
Without loss of generality, suppose that $n$ and $m$ are nonnegative integers such that $n > m$. Let $(\xio{n},\theta_n^0)\in (\ksi{n}\times S^1)/\sim$ and $(\xio{m},\theta_m^0)\in (\ksi{m}\times S^1)/\sim$ be two initial conditions of the orbits $\orbitksi{n}$ and $\orbitksi{m}$, respectively, where we are assuming $\theta_n^0$ and $\theta_m^0$ are both inward pointing.  If $\theta_n^0 = \theta_m^0$ and if $\xio{n}$ and $\xio{m}$ lie on a segment determined from $\theta_n^0$ (or $\theta_m^0$) that intersects $\ksi{n}$ only at $\xio{n}$, then we say $(\xio{n},\theta_n^0)$ and $(\xio{m},\theta_m^0)$ are \textit{compatible initial conditions}.
\end{definition}

\begin{remark}
When two initial conditions $(\xio{n},\theta_n^0)$ and $(\xio{m},\theta_m^0)$ are compatible, then we simply write each as $(\xio{n},\theta^0)$ and $(\xio{m},\theta^0)$.
\end{remark}

Not every orbit must pass through the region of $\omegaksi{n}$ corresponding to the interior of $\omegaksi{0}$, let alone pass through the interior of $\omegaksi{m}$, for any $m<n$.  Because of this, it may be the case that an initial condition $(\xio{n},\theta^0)$ is not compatible with $(\xio{m},\theta^0)$, for any $m<n$.

\begin{definition}[Sequence of compatible initial conditions]
\label{def:sequenceOfCompatibleInitialConditions}
Let $\{(\xio{i},\theta_i^0)\}_{i=N}^\infty$ be a sequence of initial conditions, for some integer $N\geq 0$.  We say that this sequence is a \textit{sequence of compatible initial conditions} if for every $m\geq N$ and for every $n> m$, we have that $(\xio{n},\theta_n^0)$ and $(\xio{m},\theta_m^0)$ are compatible initial conditions.  In such a case, we then write the sequence as $\{(\xio{i},\theta^0)\}_{i=N}^\infty$.
\end{definition}

\begin{definition}[Sequence of compatible orbits]
\label{def:SequenceOfCompatibleOrbits}
Consider a sequence of compatible initial conditions $\{(\xio{n},\theta^0)\}_{n=N}^\infty$.  Then the corresponding sequence of orbits $\compseqi{N}$ is called \textit{a sequence of compatible orbits}.
\end{definition}

If $\orbitksixang{m}{\yio{m}}{\theta(\varpi_m^0)}$ is an orbit of $\omegaksi{m}$, then $\orbitksixang{m}{\yio{m}}{\theta(\varpi_m^0)}$ is a member of a sequence of compatible orbits $\compseqiang{N}{\varpi^}$ for some $N\geq 0$. It is clear from the definition of a sequence of compatible orbits that such a sequence is determined by the first orbit $\orbitksiang{N}{\varpi^0}$.  Since the initial condition of an orbit determines the orbit, we can say without any ambiguity that a sequence of compatible orbits is determined by an initial condition $(\xio{N},\varpi^0)$.

\begin{definition}[A sequence of compatible $\mathcal{P}$ orbits]
Let $\mathcal{P}$ be a property (resp., $\mathcal{P}_1,...,\mathcal{P}_j$ a list of properties).  If every orbit in a sequence of compatible orbits has the property $\mathcal{P}$ (resp., a list of properties $\mathcal{P}_1,...,\mathcal{P}_j$), then we call such a sequence \textit{a sequence of compatible $\mathcal{P}$ (resp., $\mathcal{P}_1,...,\mathcal{P}_j$)  orbits}.
\end{definition}



We know that for each fixed billiard table $\omegaksi{n}$ and fixed direction $\theta_n^0$, an orbit is either closed or dense, regardless of the initial basepoint $\xio{n}$.  Applying the results in \S\S\ref{sec:SKSnAsABranchedCover} and \ref{subsec:minimalityOfTheFlowOnOmegaKSnAndConsequences}, we have the following.

\begin{theorem}[A topological dichotomy for sequences of compatible orbits]
\label{thm:ATopologicalDichotomy}
Let $\compseqi{N}$ be a sequence of compatible orbits.  Then $\compseqi{N}$ is either entirely comprised of closed orbits or is entirely comprised of dense hybrid orbits.
\end{theorem}

\begin{proof}
Let $\compseqi{N}$ be a sequence of compatible orbits.  By construction $\theta^0$ is the same initial direction for every orbit in the sequence of compatible orbits.  Suppose $\theta^0$ is rational with respect to the basis $\{u_1,u_2\}:=\{(1,0),(1/2,\sqrt{3}/2)\}$.  Then, applying Corollary \ref{cor:gutkinsResultAppliedToPrefractals}, for every $n\geq N$, we deduce that the orbit $\orbitksiang{n}{\theta^0}$ is a closed orbit of $\omegaksi{n}$.  Hence, $\compseqi{N}$ is a sequence of compatible orbits for which every orbit in the sequence is closed.

Suppose now that $\theta^0$ is irrational with respect to the basis $\{u_1,u_2\}$.  Then, by Corollary \ref{cor:gutkinsResultAppliedToPrefractals}, for every $n\geq N$, the orbit $\orbitksi{n}$ is a dense orbit of $\omegaksi{n}$.  By Proposition \ref{prop:denseOrbitIsADenseHybridOrbit}, $\orbitksi{n}$ is therefore a dense hybrid orbit of $\omegaksi{n}$.  Hence, $\compseqi{N}$ is a sequence of compatible orbits for which every orbit in the sequence is a dense hybrid orbit.
\end{proof}






\begin{theorem}
\label{thm:hybridOrbitOfKS0ImpliesHybridOrbitOfKSn}
If $\orbitksi{0}$ is a periodic hybrid orbit of $\omegaksi{0}$ with no basepoints corresponding to ternary points, then for every $n\geq 0$, the compatible orbit $\orbitksi{n}$ is a periodic hybrid orbit of $\omegaksi{n}$.
\end{theorem}

\begin{proof}
Since $\omegaksi{n}$ can be tiled by scale $n$ copies of $\omegaksi{0}$, an orbit of $\omegaksi{0}$ can be unfolded in the Koch snowflake prefractal billiard $\omegaksi{n}$, for every $n\geq 0$;  see \S\ref{subsec:UnfoldingABilliardOrbit}. Therefore, each basepoint of the corresponding compatible orbit $\orbitksi{n}$ will have a ternary representation consistent with that described in Definition \ref{def:hybridOrbit}.
\end{proof}

\begin{theorem}[A sequence of compatible periodic hybrid orbits]
Consider a vector $(a,b)$ that is rational with respect to the basis $\{u_1,u_2\}:=\{(1,0),(1/2,\sqrt{3}/2)\}$ and let $\xoo\in I$.  Then, we have the following\emph{:}

\begin{enumerate}
\item{If $a$ and $b$ are both positive integers with $b$ being odd, $\xoo = \frac{r}{4^s}$, for some $r,s\in\mathbb{N}$ with $s\geq 1$, $1\leq r<4^s$ being odd  and $\theta^0 := \arctan{\frac{b\sqrt{3}}{2a+b}}$, then the sequence of compatible closed orbits $\compseq$ is a sequence of compatible periodic hybrid orbits.}

\item{If $a=1/2$, $b$ is a positive odd integer, $\xoo = \frac{r}{2^s}$, for some $r,s\in\mathbb{N}$ with $s\geq 1$, $1\leq r<2^s$ being odd and $\theta^0 := \arctan{\frac{b\sqrt{3}}{2a+b}}$, then the sequence of compatible closed orbits $\compseq$ is a sequence of compatible periodic hybrid orbits.}
\end{enumerate}
\label{thm:bodd}
\end{theorem}

\begin{proof}
Let $r,s\in \mathbb{N}$, with $s\geq 1$ and $1\leq r < 4^s$, $a$ and $b$ both be positive integers with $b$ being odd and $\xoo = \frac{r}{4^s}$.  Suppose a line starting at $(\xoo, 0)$ with slope $\frac{b\sqrt{3}}{2a+b}$ intersects a point in $\mathbb{R}^2$ that would correspond to a lattice point of a lattice comprised of equilateral triangles at scale $k$. If $m,n,p,q,k\in\mathbb{Z}$, with $k\geq 1$ and $p,q\leq 3^k$, then such a point has the form $(m+p/3^k)u_1 + (n+q/3^k)u_2$.  Then, using the equation for a line in the plane, we find that
\begin{align}
\left(n+\frac{q}{3^k}\right)\frac{\sqrt{3}}{2} &= \frac{b\sqrt{3}}{2a+b}\left(m+\frac{p}{3^k}+\frac{n}{2}+\frac{q}{2\cdot 3^k}-\frac{r}{4^s}\right),\\
\left(\frac{3^kn+q}{3^k}\right)\frac{1}{2} &= \frac{b}{2a+b}\left(\frac{4^s3^km+4^sp+2\cdot 4^{s-1}3^kn + 2\cdot 4^{s-1}q-3^k r}{3^k4^s}\right),\\
2\cdot 4^{s-1}(3^kn+q)(&2a+b) = b(4^s3^km+4^sp+2\cdot 4^{s-1}3^kn + 2\cdot 4^{s-1}q-3^k r).
\label{eqn:evenNotOdd}
\end{align}
Since $b$ and $r$ are odd, the left-hand side of Equation (\ref{eqn:evenNotOdd}) is even, but the right-hand side is not.  Therefore, our assumption that such a point corresponding to a lattice point at scale $k$ laid on the line beginning at $\xoo=r/4^s$ with slope $\frac{b\sqrt{3}}{2a+b}$ was incorrect.  It follows that such a line emanating from $\xoo=r/4^s$  avoids all points in the boundary of $\omegaksi{0}$ having finite ternary representations. By Theorem \ref{thm:hybridOrbitOfKS0ImpliesHybridOrbitOfKSn}, every orbit in the sequence of compatible orbits must therefore be a periodic hybrid orbit, meaning that $\compseq$ is a sequence of compatible periodic hybrid orbits.

If $a=1/2$, $b$ is a positive odd integer, $\xoo=r/2^s$, with $s\geq 1$ and $1\leq r<2^s$, then a similar argument shows that $\compseq$ is a sequence of compatible periodic hybrid orbits. Suppose a line starting at $(\xoo, 0)$ with slope $\frac{b\sqrt{3}}{2a+b}$ intersects a point in $\mathbb{R}^2$ that would correspond to a lattice point of a lattice comprised of equilateral triangles at scale $k$. If $m,n,p,q,k\in\mathbb{Z}$, with $k\geq 1$ and $p,q\leq 3^k$, then such a point has the form $(m+p/3^k)u_1 + (n+q/3^k)u_2$.  Then, using the equation for a line in the plane, we find that
\begin{align}
\left(n+\frac{q}{3^k}\right)\frac{\sqrt{3}}{2} &= \frac{b\sqrt{3}}{2a+b}\left(m+\frac{p}{3^k}+\frac{n}{2}+\frac{q}{2\cdot 3^k}-\frac{r}{2^s}\right),\\
\left(\frac{3^kn+q}{3^k}\right)\frac{1}{2} &= \frac{b}{2a+b}\left(\frac{2^s3^km+2^sp+2^{s-1}3^kn + 2^{s-1}q-3^k r}{3^k2^s}\right),\\
2^{s-1}(3^kn+q)(&2a+b) = b(2^s3^km+2^sp+2^{s-1}3^kn + 2^{s-1}q-3^k r).
\label{eqn:evenNotOddII}
\end{align}
\noindent Since $b$ and $r$ are odd and $a=1/2$, we see that the left-hand side of Equation (\ref{eqn:evenNotOddII}) is even and the right-hand side is not. Therefore, our assumption that such a point corresponding to a lattice point at scale $k$ laid on the line beginning at $\xoo=r/2^s$ with slope $\frac{b\sqrt{3}}{2a+b}$ was incorrect.  It follows that such a line emanating from $\xoo=r/2^s$  avoids all points in the boundary of $\omegaksi{0}$ having finite ternary representations. By Theorem \ref{thm:hybridOrbitOfKS0ImpliesHybridOrbitOfKSn}, every orbit in the sequence of compatible orbits must therefore  be a periodic hybrid orbit, meaning that $\compseq$ is a sequence of compatible periodic hybrid orbits.
\end{proof}

\begin{example}[A sequence of compatible periodic hybrid orbits]
\label{exa:ASequenceOfCompatiblePeriodicHybridOrbits}
In Figure \ref{fig:hybrid012}, three periodic hybrid orbits are displayed.  These three orbits constitute the first three terms in a sequence of compatible periodic hybrid orbits.  If we choose $\xoo=\overline{c}\in I$ and $\theta_0^0$ to be an angle such that $\xoo$ connects with the midpoint of the lower one-third interval on the side of $\omegaksi{0}$, we can see that $\orbitksi{0}$ is a periodic hybrid orbit.  More importantly, there are elements of the footprint $\footprintksi{0}$ with ternary representations of type $\tern{lr}{c}$.  This observation is key for constructing what we call nontrivial polygonal paths of $\omegaks$, a topic which is discussed in more detail in \nolinebreak\S\ref{sec:nontrivialPolygonalPaths}.

\begin{figure}
\begin{center}
\includegraphics[scale=.25]{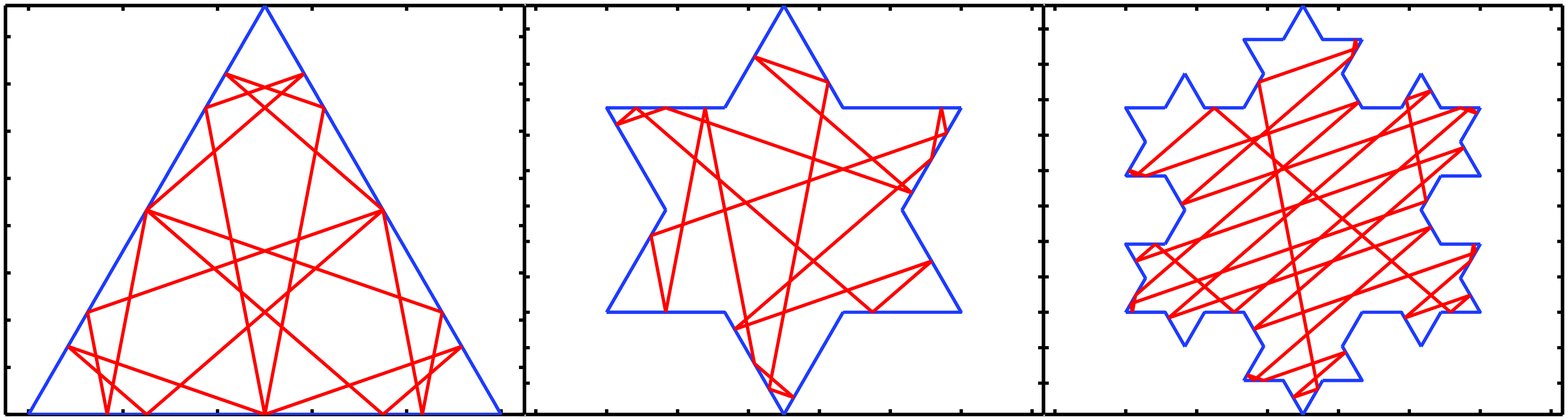}
\end{center}
\caption{Three examples of periodic hybrid orbits.  These are the first three elements of the sequence of compatible periodic hybrid orbits described in Example \ref{exa:ASequenceOfCompatiblePeriodicHybridOrbits}.}
\label{fig:hybrid012}
\end{figure}

\end{example}



\begin{example}[A sequence of compatible hook orbits]
\label{exa:ASequenceOfCompatibleHookOrbits}
Let $\xoo\in I$ have a ternary representation given by $\overline{rl}$, which is a Cantor-point of $\ks$ (in the sense of \S\ref{subsec:TheKochSnowflake}). Such a point has a value of $3/4$.  Considering an orbit of $\omegaksi{0}$ with an initial direction of $\pi/6$, the ternary representation of the basepoints at which the billiard ball path forms right angles with the sides of $\omegaksi{0}$ is of the type $\tern{c}{lr}$. This is a degenerate periodic hybrid orbit, meaning that it doubles back on itself, and the next orbit in the sequence of compatible periodic hybrid orbits has the initial condition $(\xio{1},\pi/6) = (\xoo,\pi/6)$.  Since the ternary representation of the basepoint of $f_0(\xoo,\pi/6)$ is $r\overline{c}$ and $\theta_0^0=\theta_1^0=\pi/6$, it follows that the basepoint of $f_1(\xio{1},\pi/6)$ is a Cantor-point.  Then, still in the notation introduced in Remark \ref{idx:billiardMapKSn}, the basepoint of $f^2_1(\xio{1},\pi/6)$ has a ternary representation of type $\tern{c}{lr}$.  This same pattern is repeated for every subsequent orbit in the sequence of compatible orbits.  As a result, the sequence of compatible orbits forms a sequence of orbits that is converging to a set that is well defined. That is, such a set will be some path with finite length that is effectively determined by the law of reflection in each prefractal approximation.

Such orbits are referred to as \textit{hook orbits} for the fact that they appear to be ``hooking'' into the Koch snowflake; see Figure \ref{fig:hybridpi6lv5}.
\end{example}

\begin{figure}
\begin{center}
\includegraphics[scale=0.3]{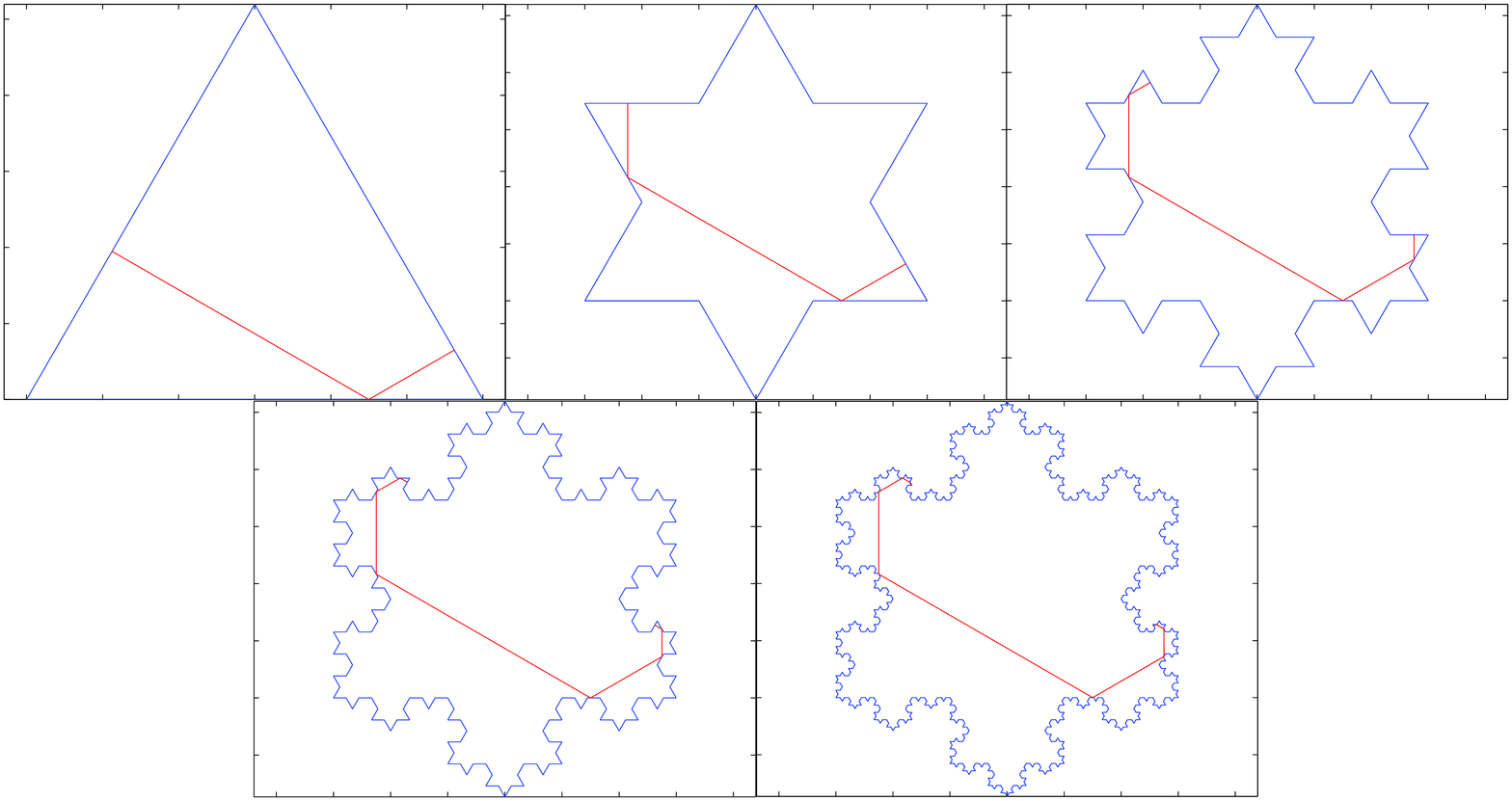}
\end{center}
\caption{An example of a hook orbit.  The same initial condition is used in each prefractal billiard.}
\label{fig:hybridpi6lv5}
\end{figure}

\begin{theorem}
\label{thm:SufficientConditionForCantorOrbit}
If every element $\xii{0}{k_0}\in \footprintksi{0}$ has a ternary representation of type $\tern{lr}{c}$, then there exists $N\geq 0$ such that $\compseqi{N}$ is a constant sequence of compatible periodic hybrid orbits.
\end{theorem}

\begin{proof}
Recall that the orbit $\orbitksi{0}$ can be unfolded in the billiard $\omegaksi{n}$.  Each element $\xii{n}{k_n}$ of the footprint $\footprintksi{n}$ of the compatible orbit $\orbitksi{n}$ has a ternary representation of type $\tern{lr}{c}$.  Since there are finitely many $c$'s in such a representation, there exists $N$ such that the ternary representation of every element $\xii{N}{k_N}\in \footprintksi{N}$ (this being the footprint of a compatible orbit $\orbitksi{N}$) is of type $\tern{lr}{\emptyset}$. As a result, the sequence of compatible periodic hybrid orbits $\compseqi{N}$ is constant, since every basepoint of every orbit remains fixed for every subsequent prefractal billiard $\omegaksi{M}$, $M\geq N$.
\end{proof}

\begin{example}[A constant sequence of compatible periodic hybrid orbits]
\label{exa:AConstantSequenceOfCompatiblePeriodicHybridOrbits}
Consider $\xoo=7/12$ in the base of the equilateral triangle. Such a value has a ternary representation of type $\tern{lr}{c}$.  Consider the initial condition $(\xoo,\pi/3)$.  The sequence of compatible orbits $\compseqiang{1}{\pi/3}$ is a constant sequence. This follows from the fact that the ternary representation of $\xio{1}$ is $\overline{rl}$. Moreover, the representation of every basepoint of $\orbitksiang{n}{\pi/3}$ is $\overline{lr}$.  In Figure \ref{fig:CompatibleCantorOrbit7-12}, we show the first three orbits in this (eventually) constant sequence of compatible periodic hybrid orbits.
\end{example}

As of now, the only examples of  constant sequences of compatible nondegenerate periodic hybrid orbits we can provide are those for which the initial direction is $\pi/3$. Of course, one can construct a constant sequence of compatible periodic hybrid orbits, each with an initial direction of $\pi/6$ or $\pi/2$, but such orbits will be degenerate.

\begin{figure}
\begin{center}
\includegraphics[scale=.75]{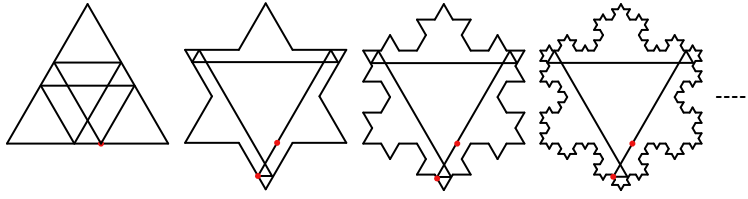}
\end{center}
\caption{An eventually constant sequence of compatible periodic hybrid orbits.  We see that the initial basepoint $\xoo=7/12$ lies on the middle third of the unit interval.  The basepoint $\xio{1}$ of the  compatible initial condition $(\xio{1},\pi/3)$ has a ternary representation of type $\tern{lr}{\emptyset}$.}
\label{fig:CompatibleCantorOrbit7-12}
\end{figure}

\section{Nontrivial polygonal paths of $\omegaks$}
\label{sec:nontrivialPolygonalPaths}

Consider a periodic hybrid orbit of $\omegaksi{0}$.  Each basepoint $\xii{0}{k_0}$ of a footprint $\footprintksi{0}$ has a ternary representation that indicates such a point never corresponds to  a ternary point of a side.  Hence, the unfolding never hits a corner of any prefractal billiard $\omegaksi{n}$.  As a result, we formulate the following conjecture.

\begin{conjecture}
\label{conj:ConvergingToAnElusiveLimitPoint}
If the sequence of basepoints $\xii{0}{k_0}$ is dynamically ordered in such a way that the type of ternary representation alternates between $\tern{c}{lr}\vee\tern{cl}{r}\vee\tern{cr}{l}\vee\tern{lcr}{\emptyset}$ and $\tern{lr}{\emptyset}$, then the corresponding sequence of compatible periodic hybrid orbits yields a sequence of basepoints converging to an elusive limit point of the Koch snowflake $\ks$.
\end{conjecture}

We know that this conjecture is true for some family of sequences of compatible periodic hybrid orbits, as evidenced by the fact that a sequence of compatible hook orbits and the sequence of compatible periodic hybrid orbits given in Figure \ref{fig:hybrid012} exhibit such a behavior. In Figure \ref{fig:MedBigBigger}, we show exactly the points referred to in Conjecture \ref{conj:ConvergingToAnElusiveLimitPoint}.  In this particular case, such points are derived from the sequence of compatible periodic hybrid orbits given in Figure \ref{fig:hybrid012}.

\begin{figure}
\begin{center}
\includegraphics[scale=.25]{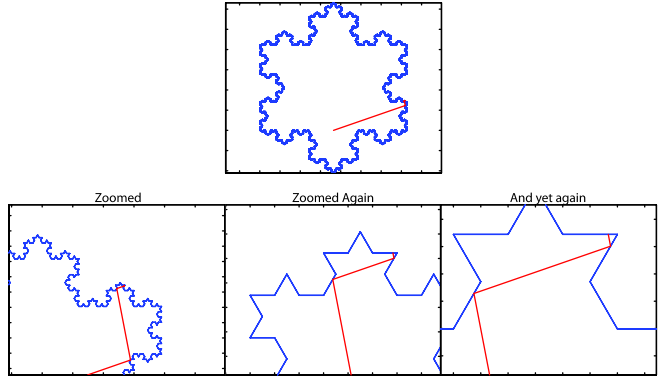}
\end{center}
\caption{A collection of basepoints from successive compatible periodic hybrid orbits converging to an elusive limit point of $\ks$.}
\label{fig:MedBigBigger}
\end{figure}

As we can see from Figure \ref{fig:MedBigBigger}, the sequences of Cantor-points determined from a sequence of compatible periodic hybrid orbits can be connected to form what we call a \textit{nontrivial polygonal path of $\omegaks$}.   Specifically, for every $n\geq 0$, there exists $N\leq n$ and a sequence of basepoints $\{\xii{n}{k_n}\}_{n=0}^N$ such that for every $j<N$, $\xii{j}{k_j}$ has a ternary representation of type $\tern{lr}{\emptyset}$ and $\xii{N}{k_N}$ has a ternary representation of type $\tern{c}{lr}\vee\tern{cl}{r}\vee\tern{cr}{l}\vee\tern{lcr}{\emptyset}$. Then each pair $\{\xii{j}{k_j}, \xii{j+1}{k_{j+1}}\}$, $0\leq j < N$, can be connected to form a line segment and, collectively, the segments form a path.  Then, $\lim_{N\to\infty} \xii{N}{k_N}$ is an elusive limit point of $\ks$ and the collection of segments $\left\{\overline{\xii{j}{k_j}\xii{j+1}{k_{j+1}}}\right\}_{j=0}^\infty$ constitutes a nontrivial polygonal path of $\omegaks$.  We denote such a path by $\mathscr{N}(\xio{0},\theta_0^0)$.

We next show how to construct  two nontrivial polygonal paths that will connect two elusive limit points of the Koch snowflake $\ks$.  Consider a sequence of compatible periodic hybrid orbits $\compseq$ that determines a nontrivial polygonal path.  Let $\overline{\theta^0} = \theta^0 +\pi$ be the angle made by a vector based at $\xii{0}{1}$ when measured relative to the fixed coordinate system and define $\overline{\xoo}:=\xii{0}{1}$.  Then the sequence of compatible hybrid periodic orbits $\compseqbar$ determines a nontrivial polygonal path of $\ks$.  Denoting the nontrivial polygonal paths of $\compseq$ and $\compseqbar$ by $\mathscr{N}(\xoo,\theta^0)$ and $\mathscr{N}(\overline{\xoo},\overline{\theta^0})$, respectively, we see that the concatenation $\mathscr{N}(\xoo,\theta^0)\cup\mathscr{N}(\overline{\xoo},\overline{\theta^0})$ determines a path from one elusive point of $\ks$ to another elusive limit point of $\ks$.

As a result, if one is given the fact that an elusive limit point $x^0\in\ks$ is the limit of a sequence of basepoints constituting the vertices of a nontrivial polygonal path, then one can determine a path from $x^0$ to another elusive limit point $x^1\in\ks$ by following the path determined by the two nontrivial polygonal paths, each being determined by the law of reflection.

\section{Concluding remarks}
\label{sec:concludingRemarks}
We have seen two extremes: sequences of compatible orbits that are (eventually) constant and sequences of compatible orbits with nontrivial limiting behavior. Ultimately, we want to answer the following question:  Can one determine the shortest path between two points of the snowflake subject to a well-defined collision in the boundary?  Finding an answer to such a question amounts to determining a suitable law of reflection in the boundary $\ks$.  As we have seen in the case of a constant sequence of compatible orbits, for certain initial conditions, it is possible to determine a well-defined orbit of the Koch snowflake.  In the case of a sequence of compatible periodic hybrid orbits that determines a nontrivial polygonal path, we have seen a way to connect two elusive limit points via two nontrivial polygonal paths of finite length; cf. \S\ref{sec:nontrivialPolygonalPaths}.  We conjecture that such paths constitute subsets of a well-defined orbit of $\omegaks$.  That is, once a suitable law of reflection is determined, we expect a nontrivial polygonal path determined from a sequence of compatible periodic hybrid orbits to be a subset of the corresponding orbit.

We note that several of the geometric and topological properties of certain possible periodic orbits (or their footprints) of $\omegaks$ are provided in \cite{LapNie2} and \cite{LapNie4}, along with some experimental evidence in support of a ``fractal law of reflection''.

Understanding how a billiard ball reflects off of an elusive limit point is at the heart of the `fractal billiards problem'.  In general, fractal snowflakes constitute the canonical examples of fractal billiards.  One may perform a similar analysis with similar results for the square snowflake.  While not the prototypical examples of a fractal billiard, the $T$-fractal and a Sierpinski carpet billiard also constitute fractal billiard tables; see Figure \ref{fig:otherFractalBilliardTables}.  Each of these tables contain elusive limit points, and in some ways may be more tractable than a snowflake billiard.  Recent work between the second author and Joe P. Chen in \cite{CheNie} extends the results of \cite{Du-CaTy} with the intention of understanding the billiard dynamics on a self-similar Sierpinski carpet billiard table.  By further understanding the nature of what are called \textit{nontrivial line segments}, one can determine orbits of a self-similar Sierpinski carpet that never intersect any corners or sides of any deleted squares of any prefractal approximation, save for those of the initial unit square.  Of course, the problem is that such orbits do intersect infinitely many elusive limit points of the self-similar Sierpinski carpet, highlighting the core problem of determining dynamics on a fractal billiard table.

Taking a different perspective, the work in progress \cite{LapNie3} seeks to understand the nature of the `\textit{fractal flat surface}' by examining the sequence of Veech groups of prefractal flat surfaces.  By extending the work of \cite{We-Sc}, the authors hope to view the prefractal flat surfaces $\mathcal{S}(\ksi{n})$ as `rhombic origamis', as opposed to the traditional square tiled surfaces found in \cite{We-Sc}.

Approaching the problem of determining dynamics on fractal billiard tables from  various perspectives may eventually bring about a clearer picture of what should constitute a true fractal billiard.  We hope that our work will help lay the foundations for this new subject and spark the interest of other researchers working on related questions and who may provide the additional insights needed for solving these difficult problems.

\begin{figure}
\begin{center}
\includegraphics[scale=.5]{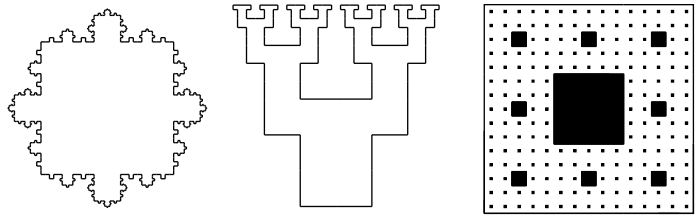}
\end{center}
\caption{The square snowflake, $T$-fractal and a Sierpinski carpet.}
\label{fig:otherFractalBilliardTables}
\end{figure}
\vspace{2 mm}
\begin{center}
\emph{Acknowledgments}
\end{center}

We would like to thank Eugene Gutkin for a helpful discussion about rational billiards and the content of his article  \cite{Gut2}.

\end{document}